\documentclass{amsart} 

\usepackage{amssymb}
\usepackage{amsmath}

\newtheorem{theorem}{Theorem}[section]
\newtheorem{lemma}[theorem]{Lemma}
\newtheorem{proposition}[theorem]{Proposition}
\newtheorem{corollary}[theorem]{Corollary}
\theoremstyle{definition}

\theoremstyle{remark}
\newtheorem{remark}[theorem]{Remark}

\numberwithin{equation}{section}

\newcommand{\R}{\mathbb{R}}

\newcommand{\Z}{\mathbb{Z}}

\renewcommand{\H}{\mathcal{H}}

\author{Mark Allen}
\address{Department of Mathematics, The University of Texas at Austin, Austin,
  TX 78712}
\email{mallen@math.utexas.edu}

\author{Wenhui Shi}
\address{Mathematisches Institut, Universit\"{a}t Bonn, Endenicher Allee 62,  53115 Bonn, Germany}
\email{wenhui.shi@hcm.uni-bonn.de}

\thanks{M.~Allen was partially supported by NSF grants DMS-1303632 and DMS-1101139. W. Shi was partially supported by Hausdorff Center of Mathematics and
NSF grant DMS-1101139.}
\thanks{We would like to thank Arshak Petrosyan for suggesting this problem.}

\title[Boundary Temperature Control Problem]{The two-phase parabolic Signorini problem}

\begin{document}

\begin{abstract}
We study solutions to a variational inequality that models heat control on the boundary. This problem can be thought of as the two-phase parabolic Signorini problem. Specifically, we study variational solutions to the inequality
\[   
\begin{aligned}
&\int_{\Omega_T}(\partial_t u)(w-u) +  \nabla u  \nabla (w-u)\ dx dt  \\
&+ \int_{S}{\lambda_+ (w^+ - u^+) + \lambda_- (w^- - u^-)  } \ d\H^{n-1} d t \geq 0 
\end{aligned}
\]
without any sign restriction on the function $u$. The main result states that the two free boundaries 
(in the topology of $S:= \partial \Omega \times (0,T)$)
\[
\Gamma^+ = \partial \{u> 0\} \cap S  \text{ and }
\Gamma^- = \partial \{u< 0\} \cap S 
\]
cannot touch. i.e. $\Gamma^+ \cap \Gamma^- = \emptyset$, therefore reducing the study of the free boundary to the parabolic Signorini problem. The separation also allows us to show the optimal regularity of the solutions. 
\end{abstract}

\maketitle
\section{Introduction}    \label{s:introduction}
\subsection{Background and main results}
In this paper we study a parabolic thin obstacle type equation
\begin{align}
&\Delta u-\partial_t u =0 \quad \text{ in } \Omega_T:=\Omega \times (0,T),\notag\\
&-\partial_\nu u\in \left\{
\begin{matrix}
\lambda_+ &\text{ if } u>0\\
[-\lambda_-,\lambda_+] &\text{ if } u=0 \\
-\lambda_- &\text{ if } u<0 
\end{matrix} \right. 
\quad \text{ on } S:=\partial \Omega \times (0,T),\label{e:sig0}\\
&u(\cdot, 0)= u_0(x) \quad \text{ on } \Omega,\notag
\end{align}
where $\Omega $ is a bounded domain in $\R^n$ ($n\geq 2$) with $C^2$ boundary, $\nu$ is the exterior unit normal, $\lambda_+$, $\lambda_-$ are positive constants and $u_0$ is a given $C^2$ function.

This problem arises as a limiting case in modeling the heat control on the \emph{boundary} \cite{dl76}. Specifically, consider the problem of maintaining the temperature on 
the boundary $\partial \Omega$ in the fixed interval $[h_1, h_2]$. If the temperature rises above $h_2$ or below $h_1$
at a point $x \in \partial \Omega$
we cool or heat the domain by injecting a fixed quantity of heat at $x$. A similar problem was recently studied in \cite{ac10}. Our problem formulated in \eqref{e:sig0} corresponds to the limiting case when $h_1=h_2$ with instantaneous heat control.  


There are several relevant reasons for studying the specific limiting case in which $h_1=h_2$. The problem in \eqref{e:sig0} can be seen as the ``thin'' version of the two-phase parabolic obstacle problem 
\begin{equation}\label{e:thick}
  \Delta u-\partial_t u=\lambda_+\chi_{\{u>0\}}-\lambda_-\chi_{\{u<0\}} \text{ in } \Omega_T, 
\end{equation}
which arises as a limiting case for a model for heat control regulated through the \emph{interior} of the domain \cite{dl76}. Then in both mathematical formulation and application, \eqref{e:sig0} can be considered the ``parabolic two-phase thin obstacle problem''.

In the two-phase problem, it is interesting to study the interaction between the free boundaries coming from the two phases, i.e. $\Gamma^\pm=\partial\{ \pm u>0\}$ (or $\Gamma^\pm =\partial\{\pm u>0\}\cap S$ in the ``thin" version). The ``branch point" of the free boundaries of \eqref{e:thick} was studied in \cite{SNW09}.
In this paper, we show that in the ``thin" case there is no ``branch point" of the free boundaries. More precisely, we prove the following:
\begin{theorem}\label{t:sep}
Let $u$ be a solution to \eqref{e:sig0}. Then the free boundaries of the two phases do not intersect, i.e.
$\Gamma^+\cap \Gamma^-=\emptyset$.
\end{theorem}

The above theorem states that even if the positivity and negativity phases touch at the initial time, there is an immediate separation of the free boundaries as $t>0$. We emphasize that such separation is in general not true in the classical two-phase obstacle type problem (see \cite{SW06}, \cite{SNW09b} for the stationary case and \cite{SNW09} for the evolutionary case), where the free boundary is in the \emph{interior} of the domain and of codimension one.

Recently, the elliptic (time-independent) version of this problem was studied in \cite{alp12}. The authors in \cite{alp12} showed that indeed the separation of the free boundaries is preserved in the limiting case. 

\subsection{Consequences of separation} \label{s:consequence}
With the separation of the free boundaries at hand, one can locally reduce \eqref{e:sig0} to the \emph{parabolic Signorini problem} (also called one-phase thin obstacle problem): 
\begin{align}
&\Delta u-\partial_t u=f \text{ in } \Omega_T;\label{e:sig1}\\
& u\geq \psi, \quad \partial_\nu u \geq 0, \quad (u-\psi)\partial_\nu u=0 \text{ on } S,\label{e:sig2}
\end{align}
where $\psi$ is a given $C^2$ function on $\partial \Omega$ and $f$ is bounded.  

The reduction can be done as follows. Let $u$ be a solution to \eqref{e:sig0} and $(x,t)\in \Gamma^+$ (or $\Gamma^-$). Due to the separation of the two phases, there is a cylindrical neighborhood $\mathcal{Q}$ of $(x,t)$ such that $u\geq 0$ on $\mathcal{Q}\cap S$. Let $v=u-\lambda_+ d(x)$, where $d(x)=dist((x,t),\mathcal{Q}\cap S )$ for $(x,t)\in \mathcal{Q}$. Since $d(x)=0$, $\nabla d(x)=-\nu_x$ on $\mathcal{Q}\cap S$, then it easy to see that $v$ is a local solution to the parabolic Signorini problem with $f=-\lambda_+\Delta d$ and $\psi=0$. Moreover, the free boundary remains the same under this reduction. 

The separation of the free boundary has two main consequences. The first consequence is that the study of the local properties of the free boundaries of \emph{two-phase boundary temperature control problem} \eqref{e:sig0} is completely reduced to studying the free boundary in the \emph{parabolic Signorini problem} \eqref{e:sig1}-\eqref{e:sig2}. 
For the latter problem, when the free boundary is on a flat portion of $\partial \Omega\times (0,T)$, the optimal regularity of the solution and the structure of the free boundary are studied in \cite{dgpt11}. Therefore, as a corollary to Theorem \ref{t:sep}, when $\Gamma^\pm$ are contained on a flat portion of $\partial \Omega \times (0,T)$, we obtain results for the structure 
of the free boundary. 

The second consequence of the separation is that we obtain immediate optimal  regularity for the solutions
(see Corollary \ref{c:optimal}) when $\partial \Omega \times (0,T)$ is flat, and H\"older continuity of spatial 
gradients when $\partial \Omega \times (0,T)$ is not necessarily flat. 
H\"older regularity for solutions can be obtained through a standard argument
(see section \ref{s:holder}). However, 
obtaining higher regularity - namely, H\"older continuity for the spatial derivatives - is much more difficult and 
technical. 
For the parabolic Signorini problem, the regularity of solutions was studied in \cite{u85} and \cite{au96} where a
method for obtaining H\"older continuity of the gradient is given. Optimal regularity of solutions to the parabolic
Signorini problem was recently proven in \cite{dgpt11}. Rather than attempt to adapt and utilize the arguments in 
\cite{u85}, \cite{au96}, and \cite{dgpt11} to a two-phase problem, we obtain the same regularity results quickly 
by first proving the separation of the phases and reducing the regularity 
problem to that of a one-phase problem. 

\subsection{Further remarks}
One may notice that the separation of the positivity and negativity phases would be an immediate consequence 
if the H\"older continuity of gradients was a priori known. 
One contribution of this paper is providing a method for two-phase problems where 
a certain sufficient regularity of solutions would imply a separation of the phases but proving such sufficient regularity could be lengthy and difficult. By first proving a separation, one may reduce the problem (locally) 
to a one-phase problem and thus not have to repeat the arguments for the two-phase problem. Our paper presents a method for proving a separation of the phases that is relatively short and simple and does not require first obtaining
H\"older continuity of spatial derivatives. 

Our method can be easily generalized to more general parabolic operators
$L=-\partial_t+\partial_i(a^{ij}\partial_j) +b^i\partial_i +c$ with $a^{ij}\in C^{0,1}(\overline{\Omega})$ positive definite, $b^i, c\in L_\infty(\Omega)$ and $c\geq 0$. All the results before Theorem~\ref{c:uniform} will still hold true.

\subsection{Outline of the paper}
The outline of this paper is as follows.

- In section \ref{s:notation} we provide the notation that will be used throughout the paper. 



- In section \ref{s:holder} we outline the existence, uniqueness and Lipschitz regularity of the solutions in the sense of variational inequalites.  

- In section \ref{s:nondegeneracy} we prove a nondegeneracy estimate that states that solutions must grow by a certain factor away from free boundary points. 



- In section \ref{s:separation} we use the results from the previous sections in combination with a monotonicity formula to provide a simple proof of Theorem \ref{t:sep}. We also state as a consequence results involving the optimal regularity and the free boundary. 


\section{Notation}  \label{s:notation}
We will follow the notation used in \cite{dgpt11}. For a point $x \in \R^n$ we denote $x=(x',x_n)$ where $x'=(x_1, \ldots, x_{n-1})$. For $r>0$, $x_0\in \R^n$, $t_0\in \R$ we let
\begin{alignat*}{3}
&B_r(x_0) &&=\{x\in\R^n\mid |x| < r\} &&\text{ (Euclidean ball)}\\
&B^\pm_r(x_0) &&= B_r(x_0)\cap \R^n_\pm &&\text{ (Euclidean halfball)}\\
&B'_r(x_0) &&=B_r(x)\cap \R^{n-1} &&\text{ (``thin'' ball)}\\
&Q_r(x_0,t_0) &&= B_r(x_0)\times(t_0-r^2,t_0] &&\text{ (parabolic cylinder)}\\
&Q_r'(x_0,t_0) &&= B_r'(x_0)\times(t_0-r^2,t_0] &&\text{ (``thin'' parabolic cylinder)}\\
&Q_r^\pm(x_0,t_0) &&= B_r^\pm(x_0)\times(t_0-r^2,t_0] &&\text{ (parabolic halfcylinders)}\\
&\widetilde Q_r(x_0,t_0) &&= B_r(x_0)\times(t_0-r^2,t_0+r^2) &&\text{ (full parabolic cylinder)}\\
&\widetilde Q_{r}'(x_0,t_0) &&= B_{r}'(x_0)\times(t_0-r^2,t_0+r^2) &&\text{ (full ``thin'' parabolic cylinder)}
\end{alignat*}
For simplicity we omit the center if it is the origin.

For $(x,t),(y,s)\in \R^{n+1}$, $d_p((x,t),(y,s))=\max\{|x-y|,|t-s|^{1/2}\}$ is the parabolic distance; $d_p(E_1,E_2)$ is the parabolic distance between two subsets $E_1, E_2\in \R^{n+1}$.

The parabolic Sobolev spaces $W^{2m,m}_2(\Omega_T)$, $m\in \Z^+$, are the Banach spaces of functions with a generalized derivative $\partial _x^\alpha \partial _t^ju\in L_2(\Omega_T)$ for $|\alpha |+2j\leq 2m$ and the norm
$$\|u\|_{W^{2m,m}_2(\Omega_T)}=\sum_{|\alpha |+2j\leq 2m} \|\partial _x^\alpha \partial _t^ju\|_{L_2(\Omega_T)}.$$
Parabolic Sobolev spaces $W^{1,0}_2(\Omega_T)$, $W^{1,1}_2(\Omega_T)$ are the Banach spaces with
\begin{align*}\|u\|_{W^{1,0}_2(\Omega_T)}&=\|u\|_{L_2(\Omega_T)}+\|\nabla u\|_{L_2(\Omega_T)}\\
\|u\|_{W^{1,1}_2(\Omega_T)}&=\|u\|_{L_2(\Omega_T)}+\|\nabla u\|_{L_2(\Omega_T)}+\|\partial _tu\|_{L_2(\Omega_T)}.
\end{align*}
Parabolic H\"older spaces $H^{\alpha, \alpha/2}(\Omega_T)$ for $\alpha \in (0,1]$ is the set of continuous functions which are H\"older-$\alpha$ in space and H\"older-$\alpha/2$ in time.

\section{Existence, Uniqueness and Lipschitz regularity} \label{s:holder}
We first outline the existence and uniqueness of the weak solution of \eqref{e:sig0} in the sense of variational inequalities (see \cite{dl76} for more details). 
We say $u\in W^{1,1}_2(\Omega_T)$ is a solution of \eqref{e:sig0} if it satisfies
\begin{multline}\label{e:vari}
\int_{\Omega_T} \partial_t u (w-u)+\nabla u \nabla(w-u) +
 \int_{S}\mathcal{B}(w)-\mathcal{B}(u) \geq 0,\quad \forall w\in W^{1,0}_2(\Omega_T),\\
\text{ where }\mathcal{B}(s)=\lambda_+s\chi_{\{s\geq 0\}}-\lambda_-s\chi_{\{s<0\}},\quad u(\cdot, 0)=u_0.
\end{multline}

Existence of solution to \eqref{e:vari} can be shown by a standard approximation approach. For $\epsilon>0$, we consider the approximation problem ($P_\epsilon$):
\begin{align*}
\Delta u_\epsilon -\partial _tu_\epsilon &=0 \text{ in } \Omega_T\\
-\partial _{\nu} u_\epsilon &=\beta_\epsilon (u_\epsilon) \text{ on } S \\
u_\epsilon(\cdot, 0) &= u_0 \text{ on } \Omega,
\end{align*}
where $\beta_\epsilon \in C^\infty (\R)$ is an approximation to $\mathcal{B}'$, such that 
\begin{equation*}
\beta_\epsilon(s)=\lambda_+ \text{ if } s\geq \epsilon;\quad \beta_\epsilon(s)=-\lambda_- \text{ if } s\leq -\epsilon;\quad \beta'_\epsilon(s)\geq 0; \quad \beta_\epsilon(0)=0.
\end{equation*}
By \cite{LSU}, there exists a unique classical solution 
$u_\epsilon $ to $(P_\epsilon)$. Moreover, $u_\epsilon$ is uniformly bounded in $W^{1,1}_2(\Omega_T)$ by using an energy estimate. 

\begin{proposition}[Existence] \label{prop:existence}
Along a sequence $\epsilon_j\rightarrow 0$, $u_{\epsilon_j}\rightarrow u$ weakly in $W^{1,1}_2(\Omega_T)$, where $u$ solves the variational inequality \eqref{e:vari}.
\end{proposition}
\begin{proof}
We only sketch the proof here. First it is not hard to verify that $u_\epsilon$ solves the following variational inequality
\begin{align*}
\int_{\Omega_T} (\partial_tu_\epsilon)(w-u_\epsilon)+\nabla u_\epsilon \nabla(w-u_\epsilon)+\int_{S} \mathcal{B}_\epsilon(w)-\mathcal{B}_\epsilon(u_\epsilon) \geq 0,
\quad \forall w\in W^{1,0}_2(\Omega_T),
\end{align*}
where $\mathcal{B}_\epsilon$ is the anti-derivative of $\beta_\epsilon$ with $\mathcal{B}_\epsilon(0)=0$. Here we crucially use the convexity of $\mathcal{B}_\epsilon$. 
Up to a subsequence, $u_{\epsilon_j} \rightharpoonup u$ in $W^{1,1}_2(\Omega_T)$ and $u_{\epsilon_j}\rightarrow u$ in $L_2(S)$ by the trace theorem (e.g. I.5.3 in \cite{dl76}). Passing to the limit in the above inequality and arguing as in section 5.6.1 of \cite{dl76}, we obtain that $u$ solves \eqref{e:sig0}.
\end{proof}

Uniqueness of the weak solution is a consequence of the following comparison principle.
\begin{proposition}[Comparison Principle] \label{p:comparison}
Let $u,v$ be two solutions with $u \leq v$ on $\partial_p \Omega_T$. Then $u \leq v$ in $\Omega_T$. 
\end{proposition}

\begin{proof}
Let $w_1 = \max \{u,v\}$ and $w_2 = \min \{u,v\}$. Then 
\begin{equation}\label{eq:equal-B}
\mathcal{B}(w_1) + \mathcal{B}(w_2) = \mathcal{B}(u) + \mathcal{B}(v),
\end{equation}
Taking $w=w_1$ in \eqref{e:sig0} for $v$ and taking $w=w_2$ in \eqref{e:sig0} for $u$, we have
\[
\int_{\Omega_T}\partial_t v (w_1-v)+\nabla v\nabla (w_1-v) +\int_{S}\mathcal{B}(w_1)-\mathcal{B}(v)\geq 0
\]
\[
\int_{\Omega_T}\partial_t u(w_2-u)+ \nabla u \nabla (w_2-u) +\int_{S}\mathcal{B}(w_2)-\mathcal{B}(u)\geq 0.
\]

Let $\theta:=u-v$. Note that $w_1-v=\theta^+$ and $w_2-u=-\theta^+$. Adding the above two inequalities we have
\[
\int_{\Omega_T} -(\partial_t \theta) \theta^+ -  \nabla \theta \nabla \theta^+  + \int_{S}\mathcal{B}(w_1)+\mathcal{B}(w_2)-\mathcal{B}(v)-\mathcal{B}(u)\geq 0,
\]
which taking account \eqref{eq:equal-B} yields
\[
\int_{\Omega_T} (\partial_t\theta)\theta^++ \nabla \theta \nabla\theta^+ =\frac{1}{2}\int_{\Omega_T}\partial _t[(\theta^+)^2]+|\nabla \theta^+|^2 \leq 0.
\]
Notice that $\theta^+(\cdot, 0)=0$, then it is not hard to see from the above inequality that $\theta^+\equiv 0$ in $\Omega_T$. 
\end{proof}

Next we show up to the boundary Lipschitz regularity of $u_\epsilon$. 

\begin{proposition}[$H^{1,1/2}$ regularity]\label{prop:Lipschitz}
Let $u_\epsilon$ be the solution of ($P_\epsilon$). Then $u_\epsilon \in H^{1,1/2}(\Omega_T)$ with the norm only depending on $n$, $\|u_0\|_{W^{1,2}(\Omega)}$ and $\Omega$. Hence, $u\in H^{1,1/2}(\Omega_T)$ if $u$ is a solution to \eqref{e:vari}.
\end{proposition}
\begin{proof}
The proof is standard. We only outline it here. 
\begin{enumerate}
\item[(1)] By the $L^\infty-L^2$ estimate of the solution we have $\|u_\epsilon\|_{L^\infty(\Omega_T)}$ is uniformly bounded. Thus $\|D^ku_\epsilon\|_{L^\infty(\widetilde{\Omega}_T)}$ for any $\widetilde{\Omega}\Subset \Omega$ is bounded by the interior gradient estimates for the heat equation.  
\item[(2)] Since $\partial \Omega$ is $C^2$, then for any $x_0\in \partial \Omega$ there exists a neighborhood $\mathcal{U}(x_0)$ and a $C^2$ diffeomorphism $T$ such that $T(\mathcal{U}(x_0)\cap \Omega)=B_r^+$ and $T(\mathcal{U}(x_0)\cap \partial \Omega)=B'_r$. Moreover, let $\tilde{x}=T(x)$, then $\partial_\nu=-\partial_{\tilde{x}_n}$ on $\mathcal{U}(x_0)\cap \Omega$. Let $\tilde{u}(\tilde{x},t)=u(T^{-1}(\tilde{x}),t)$. In the new coordinates ($P_\epsilon$) reads
\begin{alignat*}{2}
 \tilde{a}_{ij}\partial^2_{ij} \tilde{u}_\epsilon+\tilde{b}_{k}\partial_k \tilde{u}_\epsilon-\partial_t \tilde{u}_\epsilon&=0 && \quad \text{ in } B_r^+\times (0,T);\\
 \partial_{\tilde{x}_n}\tilde{u}_\epsilon&=\beta_\epsilon(\tilde{u}_\epsilon) &&\quad \text{ on } B'_r\times (0,T),
\end{alignat*}
with $\tilde{a}_{ij}(\tilde{x})\in C^1(B_r^+\cup B'_r)$ and $\tilde{b}_k(\tilde{x})\in C^0(B_r^+\cup B'_r)$. 
\item[(3)] Take the test function $\partial _{i}[(\partial _{i}\tilde{u}_\epsilon ) \xi^2]$, $i\neq n$, $\xi\in C^\infty_0(B_r^+\cup B'_r)$ in the variational formulation of the above equation. Then
$\|\partial ^2_{ij}\tilde{u}_\epsilon\|_{L_2(K)}\leq C$ for $i\neq n$ and $K\Subset (B_{r}^+\cup B'_{r})\times (0,T)$. Using the equation for $\tilde{u}_\epsilon$ we get $\|\partial^2 _{nn}\tilde{u}_\epsilon\|_{L_2(K)}\leq C$ . Next we take the test function $\partial_i \tilde{u}_\epsilon (\partial_\ell \tilde{u}_\epsilon - k)_+ \xi$ with $i, \ell\neq n$ and $k\geq \|D u_0\|_{L^\infty(\Omega)}$. By the arguments in chapter II of \cite{LSU}, we have $\|\partial_\ell \tilde{u}_\epsilon\|_{L^\infty(\tilde{K})}\leq C\|D^2 \tilde{u}_\epsilon\|_{L^2((B_{r}^+\cup B'_{r})\times (0,T))}$ for $\ell\neq n$ and $\tilde{K}\Subset (B_{r}^+\cup B'_{r})\times (0,T)$. By using the equation for $\partial_n \tilde{u}_\epsilon$ and the boundedness of $\beta_\epsilon$, we get $\|\partial_{n} \tilde{u}_\epsilon\|_{L^\infty(\tilde{K})}\leq C$. We remark here that the constant $C$ above is uniform in $\epsilon$ and we have used $\beta_\epsilon'\geq 0$. 
\item[(4)] From the maximum principle argument of Gilding, $\tilde{u}_\epsilon$ is uniformly H\"older-$1/2$ in $t$ on $B_{r'}\times (0,T)$ for each $r'<r$. The maximum principle argument can be found in e.g. Chapter II of \cite{Lieberman}. Since we are in a slightly different situation, we reproduce the proof in the Appendix.
\end{enumerate}
\end{proof}

From now on, unless otherwise stated, 
we will work on the local solutions of \eqref{e:vari} with straightened boundary. 
That is, we consider the solution $u$ to the following variational inequality
\begin{align}
\int_{Q_1^+} (\partial_t u) v+ a^{ij}(x)\partial_i u\partial_jv +b^i(x) u\partial_i v+c(x) uv+\int_{Q'_1} \mathcal{B}(u+v)-\mathcal{B}(u)\geq 0,\notag\\
\text{ for any } v\in \{ v\in W^{1,0}_2(Q_1^+): v=0  \text{ a.e. on } (\partial B_1)^+\times (-1,0]\},\label{e:localsig}
\end{align}
where the coefficient matrix $(a^{ij})\in C^{0,1}$ is positive definite and $a^{in}=0$ on $Q'_1$ for $i\neq n$, $b^i$ and $c$ are bounded, $c\geq 0$. We remark that the off-diagonal assumption, i.e. $a^{in}=0$ on $Q'_1$, it not restrictive, because it can always be achieved for positive definite Lipschitz metric $a^{ij}$ after performing a coordinate transformation (c.f \cite{Ura89}). It is not hard to see that the existence and uniqueness results for solutions of \eqref{e:localsig} are obtained via similar arguments as the constant coefficient case (i.e. Proposition~\ref{prop:existence} and Proposition~\ref{p:comparison}).
For compactness results we will later reference the following assumption for the coefficients 
 \begin{equation} \label{e:coeff} 
  \Lambda^{-1} |\xi|^2 \leq a^{ij}\xi_i \xi_j \leq \Lambda |\xi|^2 , \quad \| a^{ij} \|_{C^{0,1}} \leq \Lambda ,
  \quad \|b\|_{L_\infty},\|c\|_{L_\infty} \leq \Lambda. 
 \end{equation}
The next corollary is a consequence of Proposition~\ref{prop:Lipschitz}, which provides us with the existence of so called ''blow-ups''.

\begin{corollary} \label{c:blowup}
Let $u$ be a solution to \eqref{e:localsig}. Let 
\begin{equation}\label{e:recale}
u_r^{(x_0,t_0)}(x,t) := \frac{u(rx+x_0,r^2 t+t_0)}{r}
\end{equation}
be a rescaling of $u$ at $(x_0,t_0)$. Then there exists a sequence $u_{r_k}^{(x_0,t_0)} \to u_0$ such that 
\[
\begin{aligned}
&(1) \quad u_0 \in H^{1,1/2}(Q_R) \text{ for every } R>0 \\
&(2) \quad u_0 \text{ is a local solution to  the constant coefficient variational inequality}\\
&\quad \int_{Q_R^+}(\partial _t u_0) v+a^{ij}(x_0)\partial_i u\partial_j v+\int_{Q'_R}\mathcal{B}(u+v)-\mathcal{B}(u)\geq 0,\\
&\qquad \forall v\in W^{1,0}_2(Q^+_R),\ v=0 \text{ on } (\partial B_R)^+\times (-1,0] \text{ in } Q_R^+ \text{ for every } R>0.
\end{aligned}
\]
\end{corollary}
\begin{remark}
It is a priori not obvious if $u_0$ is zero.
\end{remark}

\section{Nondegeneracy} \label{s:nondegeneracy}
This section is devoted to proving a nondegeneracy property. This result states that the $\sup \ (\inf)$ of a solution must grow linearly from a free boundary point of $\Gamma^+ \ (\Gamma^-)$. In this section we will work on solutions with variable coefficients and flattened boundary $S$. We will work over all of $Q_1$ obtained by proper reflection, i.e. we reflect $u$, $a^{ij}$ with $i,j\leq n-1$ or $i=j=n$, $b^i$ with $i\leq n-1$ and $c$ evenly about $x_n$; reflect $a^{nj}=a^{jn}$ with $j\leq n-1$ and $b^n$ oddly about $x_n$. Note that by the off-diagonal assumption, i.e. $a^{nj}=0$ on $Q'_1$, after the reflection $a^{ij}$ is still Lipschitz continuous.

\begin{lemma}    \label{l:collapse}
There exists $\delta > 0$ depending only on $\lambda_{\pm},\Lambda$ and $n$ such that if $u^{\delta}$ is the solution of \eqref{e:localsig} with constant boundary data $\delta$ on $\partial _p Q_1$, then 
\[
u^{\delta}(x',0,t) = 0 \quad \text{ for all  } (x',0,t) \in Q_{1/2}'
\]
\end{lemma}

\begin{proof}
Suppose by way of contradiction that there exists $\delta_k \to 0$ and coefficients $a_{k}^{ij},b_k,c_k$ satisfying
\eqref{e:coeff} and 
points 
$(x_{k}',0,t_k) \in Q_{1/2}'$ such that 
\[
u^{\delta_k}(x_{k}',0,t_k) > 0
\]
By the comparison principle (Proposition \ref{p:comparison}) for local solutions, we have $u^{\delta_k} \geq 0$ in $Q_1$. Hence from Section~\ref{s:consequence}, $u^{\delta_k} - \lambda_+ x_n$ is a solution to the parabolic Signorini problem \eqref{e:sig1}-\eqref{e:sig2} (with variable coefficients) on $Q_1^+$. Then by $H^{\alpha,\alpha/2}$ estimates in $Q_{3/4}$ for $\nabla (u^{\delta_k} - \lambda_+ x_n)$ independent of $\delta$ \cite{au96}, we obtain that up to a subsequence $u^{\delta_k} \to u_0$ and $\nabla u^{\delta_k} \to \nabla u_0$ for some $u_0$ in $H^{\alpha,\alpha/2}$ in $Q_{3/4}$ for $\alpha < 1/2$. It is easy to check that $u_0 \equiv 0$ since the boundary data are zero, and solutions are unique. Now 
\[
 \partial_{x_n} u^{\delta_k}(x_{k}',0,t_k) = \lambda_+
\]
and for a subsequence $(x_{k}',0,t_k) \to (x_{0}', 0 , t_0) \in \overline{Q'}_{1/2}$, so that 
\[
 \partial_{x_n} u_{0}(x_{0}',0,t_0) = \lambda_+
\]
which is a contradiction since $u_0 \equiv 0$. 
\end{proof}

\begin{theorem}[Nondegeneracy] \label{t:nondegeneracy}
Let $u$ be a solution to \eqref{e:localsig}. There exists $\delta > 0$ with $\delta$ 
depending only on $\lambda_{\pm},\Lambda$ and $n$ such that if 
$u|_{\partial_p Q_r}  \leq  \delta r $ 
$(u|_{\partial_p Q_r} \geq  -\delta r)$ then 
\[
u(x) \leq 0 \quad (u(x) \geq 0) \qquad \text{for } x \in (Q'_{r/2} \cap \Omega_T).
\]
\end{theorem}

\begin{proof}
First we note that by rescaling we only need to prove Theorem \ref{t:nondegeneracy} on $Q_1$. By Lemma \ref{l:collapse} $u^{\delta} = 0$ in $Q_{1/2}'$ for $\delta$ sufficiently small. On the other hand, by the comparison principle (c.f. Proposition \ref{p:comparison}, which is also true in the variable coefficient case)
if $u \leq \delta$ on $\partial_p Q_1$, then $u \leq u^{\delta}$. Combining the above two facts we obtain that $u\leq 0$ in $Q'_{1/2}$ if $u|_{\partial_p Q_1}\leq \delta$ for $\delta $ sufficiently small. The case for which $u|_{\partial_p Q_r} \geq -\delta$ is proven similarly.
\end{proof}

From Theorem \ref{t:nondegeneracy} we immediately obtain the following corollary. 
\begin{corollary} \label{c:lgrowth}
If $u$ is a solution to \eqref{e:localsig} and $0 \in \Gamma^+$ $(0 \in \Gamma^-)$, then 
\begin{equation}   \label{E: 1/2 growth}
\sup_{\partial_p Q_r}u \geq C r \qquad \left( \inf_{\partial_p Q_r}u \leq -C r \right)
\end{equation}
Where $C$ depends only on $\lambda_+, \lambda_-$ and $n$.
\end{corollary}

\begin{remark}   \label{r:invariant}
All of the results in this Section may be restated with $Q_r$ replaced by the full cylinder $\widetilde{Q}_r$. The proofs will be identical.  
\end{remark}

\section{The Separation} \label{s:separation}
We begin this section by stating a monotonicity formula for parabolic equations that first appeared in \cite{c93}. Let
\[
G(x,t) := \frac{1}{(4\pi t)^{n/2}} e^{-|x|^2 /4t} \text{  for } (x,t) \in \R^n \times (0,\infty)
\] 
be the heat kernel. Then for a function $v$ and any $t>0$ define
\[
I(t,v) = \int_{-t}^{0} \int_{\R^n} |\nabla v(x,s)|^2 G(x,-s)dx ds
\]

\begin{theorem}  \label{t:monotoneg}
Let $u_1$ and $u_2$ satisfy the following conditions in the strip $\R^n \times [  -1,0 )  $
\[
\begin{aligned}
&(a) \quad \Delta u_i - \partial_t u_i \geq 0, u_i\geq 0 \\
&(b) \quad u_1 \cdot u_2 =0 \\
&(c) \quad u_1(0,0) = u_2(0,0) =0
\end{aligned}
\]
Assume also that the $u_i$ have moderate growth at infinity, for instance 
\[
\int_{B_R}{u_i^2(x,-1) \ dx} \leq Ce^\frac{|x|^2}{4+ \epsilon}
\]
For $R$ large and some $\epsilon >0$. Then 
\[
\Phi(t;u_1,u_2) := \frac{1}{t^2} I(t;u_1) I(t;u_2) 
\]
is monotone increasing for $0<t\leq 1$. 
\end{theorem}

\begin{remark}  \label{r:rescale}
If $u_r$ is defined as in \eqref{e:recale}, then
\[
\Phi(tr^2;u_1,u_2) = \Phi(t;(u_1)_r, (u_2)_r).
\]
\end{remark}

We will also utilize the case of equality for the formula in Theorem \ref{t:monotoneg}. 
\begin{proposition} \label{p:equality}
Let $u_1,u_2$ satisfy the assumptions in Theorem \ref{t:monotoneg}. \\
Then $\phi(t;u_1,u_2)$ is constant if and only if the $u_i$ are two complementary linear functions, i.e., after a rotation $u_1 = \alpha x_n^+$ and $u_2 = \beta x_n^-$ where $\alpha, \beta \geq 0$ are constants. 
\end{proposition}

\begin{proof}
The case of equality is determined by replacing all the inequalities with equalities in the proof of Theorem \ref{t:monotoneg}. The fundamental inequality in the proof of Theorem \ref{t:monotoneg} relies on a convexity property of eigenvalues. By a result of Beckner-Kenig-Pipher \cite{bkp}, equality is achieved in that instance when $u_1$ and $u_2$ are two complementary half planes passing through the origin (see discussion in \cite{ck98}). Thus if $\phi(t;u_1,u_2)$ is constant, then on each time slice $\R^n \times \{-s\}$, $u_i$ are two complementary linear functions. In the case of equality, each $u_i$ will also solve the heat equation when positive. Then each $u_i$ is time independent, and the $u_i$ are therefore two complementary linear functions. 
\end{proof}

We now proceed with the proof of Theorem \ref{t:sep}.
\begin{proof}[Theorem \ref{t:sep}]
Let $u$ be a solution to \eqref{e:localsig} in $Q_1^+$. We evenly reflect $u$ across the thin space $\R^{n-1} \times \{0\}$ and consider the solution in all of $Q_1$. Suppose by way of contradiction that there exists a point $(x_{0}',0,t_0) \in \Gamma^+ \cap \Gamma^-$ with $t_0>-1$. Moreover, we may assume $a^{ij}(x'_0,0,t_0)=\delta^{ij}$ by choosing a suitable coordinate chart. 
If $t_0<0$, we will use the results from Sections \ref{s:holder} and \ref{s:nondegeneracy} as stated for $\widetilde Q$ and the estimates that follow will be over $\widetilde{Q}$. If $t_0=0$ we will use the same results as stated for $Q$ and the estimates that follow would be stated for $Q$. We then proceed with the so called ``blow-up'' procedure. We consider the rescalings $u^{(x'_0,0,t_0)}_r$ as defined in \eqref{e:recale}. By Corollary \ref{c:blowup} 
we obtain a subsequence $u^{(x'_0,0,t_0)}_{r} \to u_0$ where $u_0$ is a solution to \eqref{e:vari} on every compact set. We will relabel $u_0=v$. By Corollary \ref{c:lgrowth} and Remark \ref{r:invariant} 
\begin{equation} \label{e:nongrowth}
\sup_{\partial_{p} \widetilde{Q}_R} u \geq CR \quad \text{and} \quad \inf_{\partial_{p} \widetilde{Q}_R} u \leq -CR
\end{equation}
for small $R$. 
Then in the limit, \eqref{e:nongrowth} will also hold for $v$ for every $0<R<\infty$. 
Also by Proposition~\ref{prop:Lipschitz} and Corollary \ref{c:blowup} we have 
\begin{equation} \label{e:boundgrowth}
|v(x,t)| \leq C(|x| + |t|^{1/2}).
\end{equation}
$v^{\pm}$ will satisfy the hypotheses for Theorem \ref{t:monotoneg}. Next we perform a blow-up on $v$. That is we consider the rescalings of $v$ with again $r \to 0$ and obtain a convergent subsequence $v_r \to v_0$. $v_0$ will be a solution to \eqref{e:vari} on every compact set and \eqref{e:nongrowth} will hold for $v_0$. By Theorem \ref{t:monotoneg}
\[
\Phi(t,v^+,v^-) 
\]
is monotone increasing for $0 < t \leq 1$. Then $\Phi(0+,v^+,v^-)$ is well defined and finite. By Remark \ref{r:rescale} we note that for $0<t \leq 1$
\[
\Phi(t,v_0^+,v_0^-)= \lim_{r \to 0} \Phi(t, v_r^+, v_r^-) = \lim_{r \to 0} \Phi(tr^2, v^+, v^-) = \Phi(0+,v^+,v^-).
\]  
Thus $\Phi(t,v_0^+,v_0^-)$ is constant. By Proposition \ref{p:equality} we conclude that $v_0^{\pm}$ are complementary linear functions. Since $v_0$ is even in the $x_n$ variable and $v_0(x',0,t)$ satisfies $\partial_{x_n}v(x',0,t)=\lambda_+$ when $v(x',0,t) \neq 0$, it follows that $v_0 = c|x_n|$ for $t \leq 0$. 

If $(x_{0}',0,t_0) \in \Gamma^+ \cap \Gamma^-$ was such that $t_0<0$, we must also show $v_0 = c|x_n|$ for $t>0$. Since $v_0$ is a solution to \eqref{e:vari} it follows that $-\lambda_- \leq c \leq \lambda_+$. Consider $w_1= (v_0 - c|x_n| )^+$ and $w_2 = -(v_0 - c|x_n|)^-$. Now $v_0 + c|x_n|$ is a solution to the heat equation when $x_n \neq 0$. Also since $-\lambda_- \leq c \leq \lambda_+$
\[
\partial_{x_n} w_i (x',0,t) \geq 0.
\]
Then by even reflection each $w_i$ is a subsolution to the heat equation with intial condition $w_i(x',x_n,0)=0$. Also each $w_i$ will also satisfy the growth estimate \eqref{e:boundgrowth}.  
It follows from the usual proofs of Tychonoff's theorem (or by bounding subsolutions from above by solutions and applying Tychonoff's theorem) that $w_i \equiv 0$, and so $v_0 \equiv c|x_n|$. This is a contradiction to $v_0$ satisfying \eqref{e:nongrowth}. 
\end{proof}

\begin{remark}  \label{r:sep}
In the above proof we actually showed that if $u$ is a solution to \eqref{e:localsig}, then there is no point $(x_{0}',0,t_0)$ such that 
\begin{equation}  \label{e:nongrowth2}
\sup_{\partial_{p} \widetilde{Q}_r(x_{0}',0,t_0)} v \geq Cr \quad \text{and} 
\quad \inf_{\partial_{p} \widetilde{Q}_r(x_{0}',0,t_0)} v \leq -Cr
\end{equation}
for every $0<r<r_0$ for some fixed $r_0$. 
\end{remark}

As a consequence of Theorem \ref{t:sep} we may obtain a uniform separation of the free boundaries based on a compactness argument.

\begin{theorem}  \label{c:uniform}
Let $u$ be a solution to \eqref{e:localsig} in $Q_1$ with
\[
\|u\|_{H^{1,1/2}(Q_1)}\leq C
\]
Then there exists $d>0$ depending on $C$ such that 
\[
d_p( (\Gamma^+ \cap Q_{1/2}) , (\Gamma^- \cap Q_{1/2})) \geq d 
\] 
\end{theorem}

\begin{proof}  
Suppose by way of contradiction that there exists a sequence of solutions $u_k$ with 
\[
d_p( (\Gamma^+(u_k) \cap Q_{1/2}) , (\Gamma^-(u_k) \cap Q_{1/2})) \to 0
\]
Up to a subsequence $u_k \to u_0$ in $H^{\alpha, \alpha/2}$ with $u_0$ a solution to \eqref{e:localsig}. Furthermore, as a consequence of Corollary \ref{c:lgrowth} it is clear that there would exist a point $(x_0,t_0) \in \overline{Q}_{1/2}$ and  $(x_0 , t_0)$ would be a point satisfying \eqref{e:nongrowth2} which is a contradiction to Remark \ref{r:sep}. 
\end{proof}

Because of the uniform separation of the free boundaries, we are able to transfer known results for solutions of \eqref{e:sig1}-\eqref{e:sig2} to solutions of \eqref{e:sig0}. In particular we may state results about the optimal regularity of solutions as well as the regularity of the free boundaries. 

\begin{corollary}   \label{c:optimal}  
Let $u$ be a solution to \eqref{e:sig0} in $Q_1^+$ with 
\[
\|u\|_{H^{1,1/2}(Q_1)}\leq C
\]
Then 
\begin{equation}  \label{e:optimal}
\|u\|_{H^{\frac{3}{2},\frac{3}{4}}(Q_{1/2}^+ \cup Q_{1/2}')} \leq C_1
\end{equation}
where $C_1$ is dependent on $C, \lambda_{\pm}$.
\end{corollary}


\begin{proof}
We begin first by defining the coincidence set
\[
\Lambda(u) := \{(x',0,t) \mid u(x',0,t) =0 \}
\]
We now consider a point $(x,t) \in Q_{1/2}$. Let $d=d_p( (\Gamma^+ \cap Q_{1/2}) , (\Gamma^- \cap Q_{1/2}))$.
If the distance from $(x,t)$ to a free boundary in $Q_{1/2}$ is greater than $d/4$ and $(x,t) \notin \Lambda(u)$, then one may use regular interior estimates for solutions to the heat equation to obtain the bound in \eqref{e:optimal} for $u$ in the cylinder $Q_{d/8}(x,t)$. If $(x,t) \in \Lambda(u)$, then if we perform an odd reflection on $u$ across the thin space $\R^{n-1}\times \{0\}$, then the reflected function $\tilde{u}$ will be a solution to the heat equation in the cylinder $Q_{d/4}$, and so again we obtain \eqref{e:optimal} for $u$ in the half cylinder $Q_{d/8}^+(x,t) \cup Q_{d/8}'(x,t)$. If the distance from $(x,t)$ to a free boundary is less than or equal to $d/4$, then either $u \mp \lambda_{\pm}x_n^+$ is a solution to \eqref{e:sig1}-\eqref{e:sig2} in $Q_{d/4}^+(x,t)$ and we utilize the optimal regularity result in \cite{dgpt11} to conclude \eqref{e:optimal} for $u$ in $Q_{d/8}^+(x,t) \cup Q_{d/8}'(x,t)$. Then by a covering argument we may conclude the result.
\end{proof}

\begin{remark}
The above regularity result is optimal since Re$(x_{n-1}+ix_n)^{3/2}+ \lambda_+ x_n$ is a time-independent solution to \eqref{e:sig0} in $Q_r(0,0)$ for $r$ sufficiently small and $\lambda_+$ sufficiently large. 
\end{remark}



\section{Appendix}\label{s:appendix}
In the Appendix, we use the maximum principle argument of Gilding to show that if a solution to a parabolic equation is Lipschitz in $x$, then it is H\"older-$1/2$ in $t$. In the interior case, this is Theorem 2.13 in \cite{Lieberman}. Since we are in a slightly different situation (up to the boundary estimate with the Neumann boundary data), for the completeness and the convenience of the reader we provide a proof.
\begin{proposition}\label{prop:reg_in_t}
Let $u\in C^{2,1}(Q_1^+\cup Q'_1)$ be a classical solution of 
\begin{align*}
a^{ij}(x,t)\partial^2_{ij}u+b_i(x,t)\partial_iu-\partial_tu=0 \text{ in }Q_1^+\\
\partial_{n}u=f \text{ on } Q'_1,
\end{align*}
where $a^{ij}$ is uniformly elliptic with 
\begin{align*}
\lambda_0|\xi|^2\leq a^{ij}(x,t)\xi_i\xi_j\leq \Lambda_0|\xi|^2, \quad \text{ for any }\xi\in \R^n \text{ and } (x,t)\in Q_1^+
\end{align*}
and it satisfies the off-diagonal condition $a^{i,n}(x',0, t)=0$ on $Q'_1$; $b^i(x,t), f(x,t)\in L^\infty(Q_1^+)$. Let $C_0:=\sum_{i=1}^n\|b^i\|_{L^\infty(Q_1^+)}+\|f\|_{L^\infty(Q_1^+)}$. Assume that for each $t\in (-1,0]$, 
\begin{align*}
\|\nabla u(\cdot, t)\|_{L^\infty(Q_1^+)}\leq L,
\end{align*}
for some $L>0$. Then there exists $C=C(n,C_0,L,\lambda_0,\Lambda_0)$ such that for each $x_0\in B_{1/2}^+\cup B'_{1/2}$,
\begin{align}\label{eq:reg_in_t}
|u(x_0,t_1)-u(x_0,t_0)|\leq C|t_1-t_0|^{1/2}, \quad t_1,t_0\in (-1,0).
\end{align}
\end{proposition}
\begin{proof}
We first show \eqref{eq:reg_in_t} for $x_0\in B'_{1/2}$.

Let $K=\Lambda_0+2C_0+2C_0^2$.  Given $t_0\in (-1,0)$ and $r\in (0,1/2)$ with $t_0+r^2/(4nK)<0$,  let
\begin{align*}
s:=\sup_{t\in (t_0,t_0+r^2/(4nK)]}|u(x_0,t)-u(x_0,t_0)|. 
\end{align*}
We want to show that $s\leq Cr$ for some $C>0$. If $s<r$, then we are done. If $s\geq r$, we consider 
$$ v^{\pm}(x,t)=\frac{2snK}{r^2}(t-t_0)+\frac{s}{r^2}|x-x_0|^2+(L+C_0 )r\pm (u-u(x_0,t_0))-C_0 x_n.$$
We will apply the maximum principle to $v^\pm$ in 
$$\widehat{Q}^+:=B_r^+(x_0)\times (t_0,t_0+r^2/(4nK)].$$
A direct computation shows that
\begin{align*}
a^{ij}\partial^2_{ij}v^{\pm}+b_i\partial_iv^{\pm} -\partial_t v^{\pm}&= -\frac{2snK}{r^2}+\frac{2s}{r^2}a^{ii}+\frac{2s}{r^2}b_i(x_i-(x_0)_i)-C_0b_n\\
&\leq -\frac{2snK}{r^2}+\frac{2sn\Lambda_0}{r^2}+\frac{2snC_0}{r}+C_0^2
\end{align*}
Using the definition of $K$ and $s\geq r$ we have
\begin{align*}
a^{ij}\partial^2_{ij}v^{\pm}+b_i\partial_iv^{\pm}-\partial_tv^{\pm}\leq 0 \text{ in }\widehat{Q}^+.
\end{align*}
On $\widehat{Q}'=B'_r(x_0)\times (t_0,t_0+r^2/(4nK)]$ we have
\begin{align*}
\partial_{n}v^\pm = \pm f - C_0 \leq 0.
\end{align*}
On $\partial_p\widehat{Q}\cap \{x_n>0\}$, if $t=t_0$ then
\begin{align*}
v^\pm(x,t)&\geq (L+C_0)r\pm (u(x,t_0)-u(x_0,t_0))-C_0x_n\\
&\geq (L+C_0)r-Lr-C_0 r\\
&\geq 0
\end{align*}
and if $|x-x_0|=r$, 
\begin{align*}
v^\pm(x,t)&\geq s+(L+C_0)r \pm ( u(x,t)-u(x_0,t)+u(x_0,t)-u(x_0,t_0))-C_0x_n\\
&\geq s + (L+C_0)r -Lr - s - C_0 r\\
&\geq 0.
\end{align*}
Hence by the maximum principle we have $v^{\pm}\geq 0$ in $\widehat{Q}^+$. Evaluating the inequality at $x=x_0$ and taking the supremum over all $t$ gives
\begin{equation*}
s\leq \frac{2snK}{r^2}\frac{r^2}{4nK}+(L+C_0)r,
\end{equation*}
which yields
\begin{equation}\label{eq:timeholder}
s\leq  2(L+C_0 )r.
\end{equation}
This implies \eqref{eq:reg_in_t} for $x_0\in B'_{1/2}$.

For $x_0\in B_{1/2}^+$, depending on $B_{r/2}(x_0)\cap \{x_n=0\}$ is empty or not, we either use the same argument as above (with $C_0x_n$ being replaced by $C_0(x_n-(x_0)_n)$), or use Theorem 2.13 in \cite{Lieberman}.
\end{proof}

\bibliographystyle{amsplain}
\bibliography{refma}

\end{document}